\documentclass[11pt,reqno]{amsart}
\usepackage{amsfonts, amssymb, amsmath, enumerate, mathtools}
\usepackage{fullpage, etoolbox}
\usepackage[usenames,dvipsnames]{xcolor}
\usepackage[linktoc=page, colorlinks, linkcolor=blue, citecolor=blue]{hyperref}
\usepackage[textsize=scriptsize]{todonotes}

\numberwithin{equation}{section}

\newcommand{\E}{\mathbb{E}}

\DeclareMathOperator{\diam}{diam}
\DeclareMathOperator{\rad}{rad}

\newcommand{\eps}{\varepsilon}

\renewcommand{\Pr}[2][]{\mathbb{P}_{#1} \left\{ #2 \rule{0mm}{3mm}\right\}}
\newcommand{\ip}[2]{\left\langle#1,#2\right\rangle}

\newcommand{\norm}[1]{\left\| #1 \right\|_2}
\newcommand{\abs}[1]{\left| #1 \right|}

\def \N {\mathbb{N}}

\def \R {\mathbb{R}}

\def \e {\varepsilon}

\def \tran {\mathsf{T}}

\def \psione {{\psi_1}}
\def \psitwo {{\psi_2}}

\newtheorem{theorem}{Theorem}[section]
\newtheorem{proposition}[theorem]{Proposition}
\newtheorem{corollary}[theorem]{Corollary}
\newtheorem{lemma}[theorem]{Lemma}

\theoremstyle{remark}
\newtheorem{remark}[theorem]{Remark}

\title[]{A simple tool for bounding \\ the deviation of random matrices on geometric sets}
\author{Christopher Liaw \and Abbas Mehrabian \and Yaniv Plan \and Roman Vershynin}
\date{\today}
\address{C.L.:  Department of Computer Science, University of British Columbia, 2366 Main Mall, Vancouver, BC V6T 1Z4, Canada}
\email{cvliaw@cs.ubc.ca}
\thanks{C.L.~is partially supported by an NSERC graduate scholarship. 
  A.M.~is supported by an NSERC Postdoctoral Fellowship.
  Y.P.~is partially supported by NSERC grant 22R23068.
  R.V.~is partially supported by NSF grant DMS 1265782 and USAF Grant FA9550-14-1-0009.}
\address{A.M.:  Department of Computer Science, University of British Columbia, 2366 Main Mall, Vancouver, BC V6T 1Z4, Canada, and
School of Computing Science, Simon Fraser University,
8888 University Drive, Burnaby, BC V5A 1S6, Canada
}
\email{abbasmehrabian@gmail.com}
\address{Y.P.:  Department of Mathematics, University of British Columbia, 1984 Mathematics Rd, Vancouver, BC V6T 1Z2, Canada}
\email{yaniv@math.ubc.ca}
\address{R.V.: Department of Mathematics, University of Michigan, 530 Church St., Ann Arbor, MI 48109, U.S.A.}
\email{romanv@umich.edu}
\begin{document}

\begin{abstract}
Let $A$ be an isotropic, sub-gaussian $m\times n$ matrix.
We prove that the process
$Z_x \coloneqq \norm{Ax} - \sqrt m \norm{x}$
has sub-gaussian increments,
that is,
$\|Z_x - Z_y\|_\psitwo \le C \|x-y\|_2$  for any  $x,y \in \R^n$.
Using this, we show that for any bounded set $T\subseteq \R^n$,
the deviation of 
$\|Ax\|_2$ around its mean is uniformly bounded by the Gaussian complexity of $T$.
We also prove a local version of this theorem,
which allows for unbounded sets. 
These theorems have various applications,
some of which are reviewed in this paper.  In particular, we give a new result regarding model selection in the constrained linear model.
\end{abstract}

\maketitle

\section{Introduction}
%===============

Recall that a random variable $Z$ is {\em sub-gaussian} if its distribution is dominated by a normal distribution. 
One of several equivalent ways to define this rigorously is to require the Orlicz norm 
$$
\|Z\|_\psitwo := \inf \big\{ K>0 :\; \E \psi_2(|Z|/K) \le 1 \} 
$$
to be finite, for the Orlicz function $\psi_2(x) = \exp(x^2) - 1$.  
Also recall that a random vector $X$ in $\R^n$ is sub-gaussian if all of its one-dimensional marginals 
are sub-gaussian random variables; this is quantified by the norm 
$$
\|X\|_\psitwo := \sup_{\theta \in S^{n-1}} \big\| \ip{X}{\theta} \big\|_{\psitwo}.
$$
For basic properties and examples of sub-gaussian random variables and vectors, see e.g. \cite{V_RMT}.

In this paper we study {\em isotropic, sub-gaussian random matrices} $A$. 
This means that we require the rows $A_i$ of $A$ to be independent, isotropic, and sub-gaussian random vectors:
\begin{equation}         \label{eq: A}
\E A_i A_i^\tran = I, \quad \|A_i\|_\psitwo \le K.
\end{equation}
In Remark \ref{rem: remove isotropic} below we show how to remove the isotropic assumption.

Suppose $A$ is an $m \times n$ isotropic, sub-gaussian random matrix, and $T \subset \R^n$ is a given set. 
We are wondering when $A$ acts as an approximate isometry on $T$, that is, 
when $\|A x\|_2$ concentrates near the value $(\E \|A x\|_2^2)^{1/2} = \sqrt{m} \|x\|_2$ uniformly over vectors $x \in T$.

Such a uniform deviation result must somehow depend on the ``size'' of the set $T$. 
A simple way to quantify the 
size of $T$ is through the {\em Gaussian complexity} 
\begin{equation}         \label{eq: gamma}
\gamma(T) := \E \sup_{x \in T} |\ip{g}{x}| \quad \text{where } g \sim N(0,I_n).
\end{equation}
One can often find in the literature the following translation-invariant cousin of Gaussian complexity, 
called the {\em Gaussian width} of $T$:
$$
w(T) := \E \sup_{x \in T} \ip{g}{x} 
= \frac{1}{2} \E \sup_{x \in T-T} \ip{g}{x}.
$$
These two quantities are closely related. Indeed, a standard calculation shows that
\begin{equation}         \label{eq: complexity vs width}
\frac{1}{3} \big[ w(T) + \|y\|_2 \big] 
\le \gamma(T) 
\le 2 \big[ w(T) + \|y\|_2 \big]
\quad \text{for every } y \in T.
\end{equation}
The reader is referred to \cite[Section 2]{PV}, \cite[Section 3.5]{V_Estimation} for other basic properties
of Gaussian width. Our main result is that the deviation of 
$\|Ax\|_2 $ over $T$ is uniformly bounded by the Gaussian complexity of $T$.

\begin{theorem}[Deviation of random matrices on sets]				\label{thm: deviation}
  Let $A$ be an isotropic, sub-gaussian random matrix as in \eqref{eq: A}, 
  and $T$ be a bounded subset of $\R^n$. Then 
  $$
  \E \sup_{x \in T}  \left| \|Ax\|_2 - \sqrt{m} \|x\|_2 \right| \le C K^2 \cdot \gamma(T).
  $$
\end{theorem}

(Throughout, $c$ and $C$ denote absolute constants that may change from line to line).
For Gaussian random matrices $A$, this theorem follows from a result of G. Schechtman \cite{Schechtman}.
For sub-gaussian random matrices $A$, one can find related results in 
\cite{KM, MPT, Dirksen}.  
Comparisons with these results can be found in Section~\ref{sec:compare}.

The dependence of the right-hand-side of this theorem on $T$ is essentially optimal.
This is not hard to see for $m=1$ by a direct calculation. For general $m$, 
optimality follows from several consequences of Theorem~\ref{thm: deviation} that are known to be sharp; 
see Section~\ref{s: random image}.

We do not know if the dependence on $K$ in the theorem is optimal or if the dependence can be improved to linear.
However, none of the previous results have shown a linear dependence on $K$ even in partial cases.

\begin{remark}[Removing isotropic condition] \label{rem: remove isotropic}
Theorem~\ref{thm: deviation} and the results below may also be restated without the assumption that $A$ is isotropic using a simple linear transformation.  Indeed, suppose that instead of being isotropic, each row of $A$ satisfies $\E A_i A_i^T = \Sigma$ for some invertible covariance matrix $\Sigma$.  Consider the whitened version $B_i := \sqrt{\Sigma^{-1}} A_i$.  Note that $\|B_i\|_\psitwo \leq ||\sqrt{\Sigma^{-1}}|| \cdot \|A_i\|_\psitwo \leq ||\sqrt{\Sigma^{-1}}|| \cdot K$.  Let $B$ be the random matrix whose $i$th row is $B_i$.  Then
\begin{align*}
\E \sup_{x \in T}  \left| \|Ax\|_2 - \sqrt{m} \|\sqrt{\Sigma} x\|_2 \right| &= 
\E \sup_{x \in T}  \left| \|B \sqrt{\Sigma} x\|_2 - \sqrt{m} \|\sqrt{\Sigma} x\|_2 \right|\\
&= \E \sup_{x \in \sqrt{\Sigma} T}  \left| \|B x\|_2 - \sqrt{m} \|x\|_2 \right|\\
&\leq C \|\Sigma^{-1}\| K^2 \gamma(\sqrt{\Sigma} T).
\end{align*}
The last line follows from Theorem \ref{thm: deviation}. Note also that 
$\gamma(\sqrt{\Sigma} T) \le \|\sqrt{\Sigma}\| \gamma(T) = \sqrt{\|\Sigma\|} \gamma(T)$, 
which follows from Sudakov-Fernique's inequality. Summarizing, 
our bounds can be extended to anisotropic distributions by including in them
the smallest and largest eigenvalues of the covariance matrix $\Sigma$.
\end{remark}

%Consider the simple case that $m=1$,
%all entries of $A$ are i.i.d.\ standard normal,
%and $T \subseteq S^{n-1}$ is nonempty.
%Consider two cases.
%If $\gamma(T) \geq 2$, we have
%$$
%  \E \sup_{x \in T}  \left| \|Ax\|_2 - \sqrt{m} \right|
%  =
%  \E \sup_{x \in T} \left| |\ip{g}{x}| - 1 \right| 
%  \geq
%  \E \sup_{x \in T} |\ip{g}{x}| -1
%  =
%  \gamma(T) - 1
%  \ge \gamma(T) / 2. 
%$$
%Next, assume $\gamma(T) < 2$.
%Let $y\in T $ be arbitrary, and let
%$g \sim N(0,I_n)$.
%Then, we have
%$$
%  \E \sup_{x \in T}  \left| \|Ax\|_2 - \sqrt{m} \right| 
%  \geq
%  \E  |\ip{g}{y}-1|
%  \geq
%  1 - \E |\ip{g}{y}|
%  = 1 - \sqrt{\frac{2}{\pi}}
%  \ge \gamma(T) / 10. 
%  $$
%
%I guess it should be optimal even for $1 \times n$ matrices, where we have $\|Ax\| = | \ip{A_1} {x}|$ and $A_1$ is a standard normal random vector. 

Our proof of Theorem~\ref{thm: deviation} given in Section~\ref{s: talagrand} is particularly simple, and is 
inspired by the approach of G.~Schechtman \cite{Schechtman}. 
He showed that for Gaussian matrices $A$, the random process 
$Z_x := \|Ax\|_2 - (\E \|Ax\|_2^2)^{1/2}$ indexed by points $x \in \R^n$,
has sub-gaussian increments, that is 
\begin{equation}         \label{eq: schechtman}
\|Z_x - Z_y\|_\psitwo \le C \|x-y\|_2 \quad \text{for every } x,y \in \R^n.
\end{equation}
Then Talagrand's Majorizing Measure Theorem implies the desired conclusion that\footnote{In this paper, we sometimes hide
  absolute constants in the inequalities marked $\lesssim$.}
$\E \sup_{x \in T} |Z_x| \lesssim \gamma(T)$.

However, it should be noted that G.~Schechtman's proof of \eqref{eq: schechtman} 
makes heavy use of the rotation invariance property of the Gaussian distribution of $A$. 
When $A$ is only sub-gaussian, there is no rotation invariance to rely on, and it was unknown 
if one can transfer G.~Schechtman's argument to this setting. 
This is precisely what we do here: we show that, perhaps surprisingly, the sub-gaussian increment 
property \eqref{eq: schechtman} holds for general sub-gaussian matrices $A$.

\begin{theorem}[Sub-gaussian process]					\label{thm: increments}
  Let $A$ be an isotropic, sub-gaussian random matrix as in \eqref{eq: A}.
  Then the random process $$Z_x := \|Ax\|_2 - (\E \|Ax\|_2^2)^{1/2} = \norm{Ax}- \sqrt{m} \|x\|_2 $$
  has sub-gaussian increments: 
  \begin{equation}\label{subgaussian_increments}
  \|Z_x - Z_y\|_\psitwo \le C K^2 \|x-y\|_2 \quad \textnormal{for every } x,y \in \R^n.
  \end{equation}
\end{theorem}

The proof of this theorem, given in Section~\ref{s: increments proof}, 
essentially consists of a couple of non-trivial applications of Bernstein's inequality; 
parts of the proof are inspired by G.~Schechtman's argument.
Applying Talagrand's Majorizing Measure Theorem (see Theorem~\ref{thm: talagrand} below),
we immediately obtain Theorem~\ref{thm: deviation}.

We also prove a high-probability version of Theorem~\ref{thm: deviation}.
%Define the radius 
%$\rad(T) := \sup_{x\in T} \norm{x}$
%and
%\begin{equation*}
%w(T) := \E \sup_{x \in T} \ip{g}{x} \quad \text{where } g \sim N(0,I_n).
%\end{equation*}
%
%This quantity is very similar, and is bounded by, the Gaussian complexity $\gamma(T)$.

%\begin{theorem}[Deviation of random matrices on sets: tail bounds]				\label{thm: deviation tail}
%  Under the assumptions of Theorem~\ref{thm: deviation}, for any $u \ge 0$ the event 
%  $$
%  \sup_{x \in T} \Big| \|Ax\|_2 - (\E \|Ax\|_2^2)^{1/2} \Big| \le C K^2 \big[  \gamma(T) + u \cdot \diam(T) \big]
%  $$
%  holds with probability at least $1-\exp(-u^2)$.
%\end{theorem}
%
%
%
%\todo[inline]{Abbas: this theorem fails when $T$ is a singleton ... we need to assume $0\in T$ or work with radius instead of diameter.}
%
%\todo[inline]{Abbas: it seems we can prove a slightly stronger statement, with $\gamma(T)$ replaced with $w(T)$ ... see Theorem~\ref{thm: deviation tail2}.}

\begin{theorem}[Deviation of random matrices on sets: tail bounds]				\label{thm: deviation tail2}
  Under the assumptions of Theorem~\ref{thm: deviation}, for any $u \ge 0$ the event 
  $$
  \sup_{x \in T}  \left| \|Ax\|_2 - \sqrt{m} \|x\|_2 \right| \le C K^2 \big[ w(T) + u \cdot \rad(T) \big]
  $$
  holds with probability at least $1-\exp(-u^2)$. Here $\rad(T) := \sup_{x\in T} \norm{x}$ denotes 
  the radius of $T$.
\end{theorem}

This result will be deduced in Section~\ref{s: talagrand} 
from a high-probability version of Talagrand's theorem. 

In the light of the equivalence \eqref{eq: complexity vs width}, notice that 
Theorem~\ref{thm: deviation tail2} implies the following simpler 
but weaker bound 
\begin{equation}         \label{eq: deviation tail 2 simple}
\sup_{x \in T}  \left| \|Ax\|_2 - \sqrt{m} \|x\|_2 \right| \le C K^2 u \cdot \gamma(T)
\end{equation}
if $u \ge 1$. Note that even in this simple bound, $\gamma(T)$ cannot be replaced with 
the Gaussian width $w(T)$, e.g.\ the result would fail for a singleton $T$. 
This explains why 
the radius of $T$ appears in Theorem~\ref{thm: deviation tail2}.

Restricting the set $T$ to the unit sphere, we  obtain the following corollary.
\begin{corollary}				\label{cor: on sphere}
   Under the assumptions of Theorem~\ref{thm: deviation}, for any $u \geq 0$ the
   event
   \[
      \sup_{x \in T \cap S^{n-1}} \left| \|Ax\|_2 - \sqrt{m} \right|
      \leq
      CK^2 \left[ w(T \cap S^{n-1}) + u \right]
   \]
   holds with probability at least $1 - \exp(-u^2)$.
   \label{thm: on sphere}
\end{corollary}

%\todo[inline]{Abbas: note that $Z_x$ is defined as $\|Ax\|_2 - \E \|Ax\|_2$ in Theorem~\ref{thm: increments},
%however the supremum of 
%$\|Ax\|_2 - (\E \|Ax\|_2^2)^{1/2}$ is studied in Theorems~\ref{thm: deviation} and~\ref{thm: deviation tail2} ... perhaps we can make them consistent, or comment that they're essentially equivalent.
%}

In Theorems~\ref{thm: deviation} and~\ref{thm: deviation tail2}, we assumed that the set $T$ is bounded.
For unbounded sets, we can still prove a `local version' of Theorem \ref{thm: deviation tail2}.
Let us state a simpler form of this result here. In Section~\ref{sec:prooflocal}, we will 
prove a version of the following theorem with a better probability bound. 

\begin{theorem}[Local version]
\label{thm: local version}
   Let $(Z_x)_{x\in \R^n}$ be a random process with sub-gaussian increments as in~(\ref{subgaussian_increments}).  Assume that the process is homogeneous, that is, $Z_{\alpha x} = \alpha Z_x$ for any $\alpha \geq 0$.
   Let $T$ be a star-shaped\footnote{Recall that a set $T$ is called star-shaped if 
   	$t \in T$ implies $\lambda t \in T$ for all $\lambda \in [0,1]$.}  
   subset of $\R^n$, and let $t \geq 1$.
   With probability at least $1-\exp(-t^2)$, we have
   \begin{equation}
      |Z_x|
      \le t\cdot CK^2 \gamma\left( T \cap \|x\|_2 B_2^n \right)  \quad \textnormal{for all } x \in T.
      \label{eqn:local_devtheorem}
   \end{equation}
\end{theorem}

Combining with Theorem~\ref{thm: increments}, we immediately obtain the following result.

\begin{theorem}[Local version of Theorem~\ref{thm: deviation tail2}]
\label{thm: local matrix version}   
   Let $A$ be an isotropic, sub-gaussian random matrix as in \eqref{eq: A}, 
   and let $T$ be a star-shaped
   subset of $\R^n$, and let $t \geq 1$.
   With probability at least $1-\exp(-t^2)$, we have
   \begin{equation}		 
      \Big| \|Ax\|_2 -  \sqrt{m} \|x\|_2  \Big|
      \le t\cdot CK^2 \gamma\left( T \cap \|x\|_2 B_2^n \right) \quad \textnormal{for all } x \in T.
      \label{eqn:local_dev}
   \end{equation}
\end{theorem}

\begin{remark}
\label{rmk: star shaped}
We note that Theorems \ref{thm: local version} and \ref{thm: local matrix version} can also  apply when $T$ is not a star-shaped set, simply by considering the smallest star-shaped set that contains $T$:
\[\text{star}(T) := \bigcup_{\lambda \in [0,1]} \lambda T.\]
Then one only needs to replace $T$ by star$(T)$ in the right-hand side of Equations \eqref{eqn:local_devtheorem} and \eqref{eqn:local_dev}.
\end{remark}

Results of the type of Theorems~\ref{thm: deviation},~\ref{thm: deviation tail2} 
and~\ref{thm: local matrix version}
have been useful in a variety of applications. 
For completeness, we will review some of these applications in the next section.

\section{Applications}

%==============

Random matrices have proven to be useful both for modeling data and transforming data in a variety of fields. 
Thus, the theory of this paper has implications for several applications.  A number of classical theoretical discoveries
as well as some new results follow directly from our main theorems.
In particular, the local version of our theorem (Theorem \ref{thm: local matrix version}), allows a new result in model selection under the constrained linear model, with applications in \textit{compressed sensing}.  We give details below.

\subsection{Singular values of random matrices}
%-------------------

The singular values of a random matrix are an important topic of study in random matrix theory.  
%Random matrices are used as models in many different fields, and so applications are broad.  
A small sample includes covariance estimation \cite{vershynin2012close}, stability in numerical analysis \cite{von1961collected}, and quantum state tomography \cite{gross2010quantum}.   

Corollary \ref{thm: on sphere} may be specialized to bound the singular values of a sub-gaussian matrix.  Indeed, take $T = S^{n-1}$ and note that $w(T) \leq \sqrt{n}$.  Then the corollary states that, with high probability, 
\[
\left| \|A x\|_2 - \sqrt{m} \right| \leq C K^2 \sqrt{n} \qquad \text{for all }  x \in S^{n-1}.
\]  
This recovers the well-known result that, with high probability, all of the singular values of $A$ reside in the interval $[\sqrt{m} - C K^2 \sqrt{n}, \sqrt{m} + C K^2 \sqrt{n}]$ (see \cite{V_RMT}).  When $n K^4 \ll m$, all of the singular values concentrate around $\sqrt{m}$.  In other words, a tall random matrix is well conditioned with high probability.

\subsection{Johnson-Lindenstrauss lemma}
%-------------------

The Johnson-Lindenstrauss lemma \cite{JL} describes a simple and effective method of dimension reduction. It shows that a (finite) set of data vectors $\mathcal{X}$ belonging to a very high-dimensional space, $\R^n$, can be mapped to a much lower dimensional space while roughly preserving pairwise distances.  This is useful from a computational perspective since the storage space and the speed of computational tasks both improve in the lower dimensional space.  Further, the mapping can be done simply by multiplying each vector by the random matrix $A/\sqrt{m}$.  

The classic Johnson-Lindenstrauss lemma follows immediately from our results.  Indeed, take $T' = \mathcal{X} - \mathcal{X}$.  To construct $T$, remove the 0 vector from $T'$ and project all of the remaining vectors onto $S^{n-1}$ (by normalizing).  Since $T$ belongs to the sphere and has fewer than $|\mathcal{X}|^2$ elements, it is not hard to show that $\gamma(T) \leq C \sqrt{\log|\mathcal{X}|}$. Then by Corollary \ref{thm: on sphere}, with high probability, 
\[\sup_{x \in T} \abs{\frac{1}{\sqrt{m}} \|A x\|_2 -  1} \leq \frac{C K^2 \sqrt{\log|\mathcal{X}|}}{\sqrt{m}}.\]
Equivalently, for all $x, y \in \mathcal{X}$
\[(1 - \delta) \|x - y\|_2 \leq \frac{1}{\sqrt{m}} \|A (x - y)\|_2 \leq (1 + \delta) \|x - y \|_2, \qquad \delta = \frac{C K^2 \sqrt{\log|\mathcal{X}|}}{\sqrt{m}}.\]
This is the classic Johnson-Lindenstrauss lemma.  It shows that as long as $m \gg K^4 \log\abs{\mathcal{X}}$, the mapping $x \rightarrow Ax/\sqrt{m}$ nearly preserves pair-wise distances.  In other words, $\mathcal{X}$ may be embedded into a space of dimension slightly larger than $\log\abs{\mathcal{X}}$ while preserving distances. 

In contrast to the classic Johnson-Lindenstrauss lemma that applies only to finite sets $\mathcal{X}$, 
the argument above based on Corollary \ref{thm: on sphere} allows $\mathcal{X}$ to be infinite. 
In this case, the size of $\mathcal{X}$ is quantified using the notion of Gaussian width
instead of cardinality.

To get even more precise control of the geometry of $\mathcal{X}$ in Johnson-Lindenstrauss lemma,
we may use the local version of our results.  
To this end, apply Theorem \ref{thm: local matrix version} combined with Remark \ref{rmk: star shaped} to the set $T = \mathcal{X} - \mathcal{X}$.  This shows that with high probability, for all $x, y \in \mathcal{X}$,
\begin{equation}
\abs{\frac{1}{\sqrt{m}} \|A (x - y)\|_2 -  \|x - y\|_2} 
\leq \frac{C K^2 \gamma \left( \text{star}(\mathcal{X} - \mathcal{X}) \cap \|x - y\|_2 B_2^n \right)}{\sqrt{m}}. \label{eq: improve JL} 
\end{equation}
One may recover the classic Johnson-Lindenstrauss lemma from the above bound using the containment $\text{star}(\mathcal{X} - \mathcal{X}) \subset \text{cone}(\mathcal{X} - \mathcal{X})$.  However, the above result also applies to infinite sets, and further can benefit when $\mathcal{X} - \mathcal{X}$ has different structure at different scales, e.g., when $\mathcal{X}$ has clusters.

\subsection{Gordon's Escape Theorem}
%--------------------

In \cite{gordon1988milman}, Gordon answered the following question:  \textit{Let $T$ be an arbitrary subset of $S^{n-1}$.  What is the probability that a random subspace has nonempty intersection with $T$?}  Gordon showed that this probability is small provided that the codimension of the subspace exceeds $w(T)$.  
This result also follows from Corollary \ref{thm: on sphere} for a general model of random subspaces. 

Indeed, let $A$ be an isotropic, sub-gaussian $m \times n$ random matrix as in \eqref{eq: A}. Then its kernel $\ker A$ 
is a {\em random subspace} in $\R^n$ of dimension at least $n-m$.
Corollary \ref{thm: on sphere} implies that, with high probability,
\begin{equation}         \label{eq: kernel escapes}
\ker A \cap T = \emptyset
\end{equation}
provided that $m \geq C K^4 w(T)^2$. To see this, note that in this case 
Corollary \ref{thm: on sphere} yields that $\left| \|Ax\|_2 - \sqrt{m} \right| < \sqrt{m}$ for all $x \in T$,
so $\|Ax\|_2 > 0$ for all $x \in T$, which in turn is equivalent to \eqref{eq: kernel escapes}.

We also note that there is an equivalent version of the above result when $T$ is a cone.  Then, with high probability,
\begin{equation}  \label{eq: cone escape}
\ker A \cap T = \{0\}  \quad \text{provided that } \quad m \geq C K^4 \gamma(T \cap S^{n-1})^2. 
\end{equation}

The conical version follows from the spherical version by expanding the sphere into a cone.

\subsection{Sections of sets by random subspaces: the $M^*$ Theorem}
%--------------------

The $M^*$ theorem \cite{milman1985geometrical, milman1985random, pajor1986subspaces} answers the following question:  \textit{Let $T$ be an arbitrary subset of $\R^n$.  What is the diameter of the intersection of a random subspace with $T$?}  We may bound the radius of this intersection (which of course bounds the diameter) using our main results, 
and again for a general model of random subspaces. 

Indeed, let us consider the kernel of an $m \times n$ random matrix $A$ as in the previous section. 
By Theorem \ref{thm: deviation tail2} (see \eqref{eq: deviation tail 2 simple}), we have
\begin{equation}         \label{eq: main whp}
\sup_{x \in T} \left| \|Ax\|_2 - \sqrt{m} \|x\|_2 \right| \le C K^2 \gamma(T)
\end{equation}
with high probability.  
On the event that the above inequality holds, we may further restrict the supremum to $\ker A \cap T$, giving
\[
\sup_{x \in \ker A \cap T} \sqrt{m} \|x\|_2  \le C K^2 \gamma(T).
\]
The left-hand side is $\sqrt m$ times the radius of $T \cap \ker A$.  Thus, with high probability,
\begin{equation}         \label{eq: M*}
\rad(\ker A \cap T) 
\leq \frac{C K^2 \gamma(T)}{\sqrt{m}}.
\end{equation}
This is a classical form of the so-called $M^*$ estimate. It is typically used for sets $T$ that contain the origin.
In these cases, the Gaussian complexity $\gamma(T)$ can be replaced by Gaussian width $w(T)$. 
Indeed, \eqref{eq: complexity vs width} with $y=0$ implies that these two quantities are equivalent.

\subsection{The size of random linear images of sets}		\label{s: random image}

Another question that can be addressed using our main results is how the size of a set $T$ in $\R^n$
changes under the action of a random linear transformation $A : \R^n \to \R^m$.
Applying \eqref{eq: deviation tail 2 simple} and the triangle inequality, we obtain 
\begin{equation}         \label{eq: random image}
\rad(AT) \le \sqrt{m} \cdot \rad(T) + C K^2 \gamma(T)
\end{equation}
with high probability. 
This result has been known for random projections, where
$A = \sqrt{n} P$ and $P$ is the orthogonal projection onto a random $m$-dimensional subspace 
in $\R^n$ drawn according to the Haar measure on the Grassmanian, see \cite[Proposition 5.7.1]{AGM}.

It is also known that the bound \eqref{eq: random image} is sharp (up to absolute constant factor) 
even for random projections, see \cite[Section 5.7.1]{AGM}. This in particular implies optimality of
the bound in our main result, Theorem~\ref{thm: deviation}.

\subsection{Signal recovery from the constrained linear model}

The constrained linear model is the backbone of many statistical and signal processing problems.  It takes the form
\begin{equation}				\label{eq: model}
y = A x + z, \qquad x \in T,
\end{equation}
where $x \in T \subset \R^n$ is unknown, $y \in \R^m$ is a vector of known observations, 
the measurement matrix $A \in \R^{m \times n}$ is known, 
and $z \in \R^m$ is unknown noise which can be either fixed or random and independent of $A$.  

For example, in the statistical linear model, $A$ is a matrix of explanatory variables, and $x$ is a coefficient vector.  It is common to assume, or enforce, that only a small percentage of the explanatory variables are significant.  This is encoded by taking $T$ to be the set of vectors with less than $s$ non-zero entries, for some $s \leq n$.  In other words, $T$ encodes \textit{sparsity}. In another example, $y$ is a vector of MRI measurements \cite{lustig2007sparse}, in which case $x$ is the image to be constructed.  Natural images have quite a bit of structure, which may be enforced by bounding the total variation, or requiring sparsity in a certain dictionary, each of which gives a different constraint set $T$.  There are a plethora of other applications, with various constraint sets $T$, including low-rank matrices, low-rank tensors, non-negative matrices, and structured sparsity. In general, a goal of interest is to estimate $x$.  

When $T$ is a linear subspace, it is standard to estimate $x$ via least squares regression, and the performance of such an estimator is well known.  However, when $T$ is non-linear, the problem can become quite complicated, both in designing a tractable method to estimate $x$ and also analyzing the performance.  The field of compressed sensing \cite{foucart2013mathematical, eldar2012compressed} gives a comprehensive treatment of the case when $T$ encodes sparsity, showing that convex programming can be used to estimate $x$, and that enforcing the sparse structure this way gives a substantial improvement over least squares regression.  A main idea espoused in compressed sensing is that random matrices $A$ give near optimal recovery guarantees. 

Predating, but especially following, the works in compressed sensing, there have also been several works which tackle the general case, giving results for arbitrary $T$ \cite{lecue2013learning, thrampoulidis2014simple, oymak2013squared, oymak2013simple, amelunxen2014living, chandrasekaran2012convex, plan2015generalized, plan2014high}.  
The deviation inequalities of this paper allow for a general treatment as well. We will first show how to recover 
several known signal recovery results, and then give a new result in Section \ref{ssec: simultaneous analysis}.

Consider the constrained linear model \eqref{eq: model}. A simple and natural way to estimate the 
unknown signal $x$ is to solve the optimization problem
\begin{equation}
\label{eq: least squares}
\hat{x} := \arg \min_{x' \in T} \|A x' - y\|_2^2
\end{equation}
%In the noiseless case where $z = 0$, the solution $\hat{x} \in T$ must satisfy $A \hat{x} = y$.
We note that depending on $T$, the constrained least squares problem \eqref{eq: least squares} may be 
computationally tractable or intractable. 
We do not focus on algorithmic issues here, but just note that $T$ may be replaced by a larger tractable
 set (e.g., convexified) to aid computation.

Our goal is to bound the Euclidean norm of the error 
$$
h := \hat{x} - x.
$$ 
%To do this, let us record a few basic properties of $h$. First, since $\hat{x}, x \in T$, 
%we have $h \in T - x \subset T - T$.
Since $\hat{x}$ minimizes the squared error, we have
$\|A \hat{x} - y\|_2^2 \leq \|A x - y\|_2^2$.
Simplifying this, we obtain 
\begin{equation}
\label{eq: noisy constraints}
\|A h\|_2^2 \leq 2 \langle h, A^T z \rangle. %\qquad \text{and} \qquad  h \in T - x \subset T - T.
\end{equation}
We now proceed to control $\|h\|_2$ depending on the structure of $T$.

\subsubsection{Exact recovery}
%--------------------

In the noiseless case where $z=0$, inequality \eqref{eq: noisy constraints} simplifies 
and we have
\begin{equation}
\label{eq: noiseless constraints}
h \in \ker A \cap (T-x).
\end{equation}
(The second constraint here follows since $h = \hat{x} - x$ and $\hat{x} \in T$.)

In many cases of interest, $T-x$ is a cone, or is contained in a cone, which is called the \textit{tangent cone} 
or \textit{descent cone}. Gordon-type inequality \eqref{eq: cone escape} then implies 
that $h = 0$, and thus we have {\em exact recovery} $\hat{x}=x$, 
provided that the number of observations $m$ significantly exceeds the Gaussian complexity of this cone: 
$m \ge C K^4 \gamma((T-x) \cap S^{n-1})^2$. 

For example, if $x$ is a sparse vector with $s$ non-zero entries, and $T$ is an appropriately scaled $\ell_1$ ball, then $T- x$ is contained in a tangent cone, $D$, satisfying $\gamma(D)^2 \leq C s\log(n/s)$.  This implies that $\hat{x} = x$ with high probability, provided $m \geq C K^4 s \log(n/s)$.

\subsubsection{Approximate recovery}
%--------------------

In the cases where $T-x$ is not a cone or cannot be extended to a narrow cone 
(for example, when $x$ lies in the interior of $T$), 
we can use the $M^*$ Theorem for the analysis of the error.
Indeed, combining \eqref{eq: noiseless constraints} with \eqref{eq: M*}, we obtain
$$
\|h\|_2 \leq \frac{C K^2 w(T)}{\sqrt{m}}.
$$
Here we also used that since $T-T$ contains the origin, we have $\gamma(T-T) \sim w(T)$ 
according to \eqref{eq: complexity vs width}.
In particular, this means that $x$ can be estimated up to an additive error of $\e$ in the Euclidean norm 
provided that the number of observations satisfies $m \ge C K^4 w(T)^2/\e^2$.

For a more detailed description of the $M^*$ Theorem, Gordon's Escape Theorem, and their implications for the constrained linear model, see \cite{V_Estimation}. 

\subsection{Model selection for constrained linear models}
%-------------------
\label{ssec: simultaneous analysis}

It is often unknown precisely what constraint set to use for the constrained linear model, and practitioners often experiment with different constraint sets to see which gives the best performance.  This is a form of model selection.  We focus on the case when the form of the set is known, but the scaling is unknown.  For example, in compressed sensing, it is common to assume that $x$ is \textit{compressible}, i.e., that it can be well approximated by setting most of its entries to 0.  This can be enforced by assuming that $x$ belongs to a scaled $\ell_p$ ball for some $p \in (0, 1]$.  However, generally it is not known what scaling to use for this $\ell_p$ ball.  

Despite this need, previous theory concentrates on controlling the error for one fixed choice of the scaling.  Thus, a practitioner who tries many different scalings cannot be sure that the error bounds will hold uniformly over all such scalings.  In this subsection, we remove this uncertainty by showing that the error in constrained least squares can be controlled simultaneously for an infinite number of scalings of the constraint set.  

Assume $x \in T$, but the precise scaling of $T$ is unknown.  Thus, $x$ is estimated using a scaled version of $T$:
\begin{equation}
\label{eq: scaled least squares}
\hat{x}_\lambda := \arg \min_{x' \in \lambda T} \|A x' - y\|_2^2, \qquad \lambda \geq 1.
\end{equation}

The  following corollary controls the estimation error.

\begin{corollary}
\label{cor:new}
Let $T$ be a convex, symmetric set.  Given $\lambda \geq 1$, let $\hat{x}_\lambda$ be the solution to \eqref{eq: scaled least squares}.  Let $h_\lambda := \hat{x}_\lambda - x$, let $v_\lambda = h_\lambda/(1 + \lambda)$, and let $\delta = \|v_\lambda\|_2$.  Then with probability at least $0.99$, the following occurs.  For every $\lambda \geq 1$, 
\begin{equation}
\label{complicated}
\delta 
\leq  \frac{C K^2 \gamma (T \cap \delta B_2^n)}{\sqrt{m}} 
	+ C K \sqrt{\frac{\gamma(T \cap \delta B_2^n) \cdot \|z\|_2}{m(1 + \lambda)}}.
\end{equation}
\end{corollary}

The corollary is proven using Theorem \ref{thm: local matrix version}.  To our knowledge, this corollary is new.  It recovers previous results that only apply to a single, fixed $\lambda$, as in \cite{plan2015generalized, lecue2013learning}.  It is known to be nearly minimax optimal for many constraint sets of interest and for stochastic noise term $z$, in which case $\|z\|_2$ would be replaced by its expected value \cite{plan2014high}.  

The rather complex bound of Equation \eqref{complicated} seems necessary in order to allow generality.  To aid understanding, we specialize the result to a very simple set---a linear subspace---for which the behaviour of constrained least squares is well known, the scaling becomes irrelevant, and the result simplifies significantly.  When $T$ is a $d$-dimensional subspace, we may bound the Gaussian complexity as $\gamma(T \cap \delta B_2) \leq \delta \sqrt{d}$.    Plugging in the bound on $\gamma(T \cap \delta B_2^n)$ into \eqref{complicated}, substituting $h_\lambda$ back in, and massaging the equation gives
\[\|h_\lambda\|_2^2 \leq C K^4 \cdot \frac{d \|z\|_2^2}{m^2} \qquad \text{as long as } \quad m \geq C K^4 d.\]
If $z$ is Gaussian noise with standard deviation $\sigma$, then it's norm concentrates around $\sqrt{m} \sigma$, giving (with high probability)
 \[\|h_\lambda\|_2^2 \leq C K^4 \cdot\frac{d \sigma^2}{m} \qquad \text{as long as } \quad m \geq C K^4 d.\]
In other words, the performance of least squares is proportional to the noise level multiplied by the dimension of the subspace, and divided by the number of observations, $m$.  This is well known. 

In this corollary,  for simplicity we assumed that $T$ is convex and symmetric.  Note that this already allows constraint sets of interest, such as the $\ell_1$ ball.  However, this assumption can be weakened.  All that is needed is for $T - \lambda T$ to be contained in a scaled version of $T$, and to be star shaped.  This also holds, albeit for more complex scalings, for arbitrary $\ell_p$ balls with $p > 0$.  

\begin{proof}[of Corollary~\ref{cor:new}]
For simplicity of notation, we assume $K \leq 10$ (say), and absorb $K$ into other constants.  
The general case follows the same proof.  
First note that $h_\lambda \in \lambda T - T$.
Since $T$ is convex and symmetric,
we have
$\lambda T - T \subset (1 + \lambda) T  $
%(The containment follows from the assumption that $T$ is convex and symmetric.)
and as $v_\lambda= h_\lambda/(1 + \lambda)$,
we get 
\begin{equation}
\label{first}
v_\lambda  \in T.
\end{equation} 
Moreover, \eqref{eq: noisy constraints} gives
\begin{equation}
\label{eq: starting v}
\|A v_\lambda\|_2^2 \leq \frac{\langle v_\lambda, A^T z \rangle }{1 + \lambda}, \qquad v_\lambda \in T.
\end{equation}
We will show that, with high probability, any vector $v_\lambda$ satisfying (\ref{first}) and (\ref{eq: starting v}) has a small norm, thus completing the proof.  We will do this by upper bounding $\langle v_\lambda, A^T z \rangle$ and lower bounding $\|A v_\lambda\|_2$ by $\|v_\lambda\|_2$ minus a deviation term.

For the former goal, let $w := A^T z/\|z\|_2$. Recall that the noise vector $z$ is fixed (and in case $z$ random and independent 
of $A$, condition on $z$ to make it fixed).  
Then $w$ is a sub-gaussian vector with independent entries whose sub-gaussian norm is upper-bounded by a constant; 
see \cite{V_RMT}.  Thus, the random process $Z_x := \langle x, w \rangle$ has the sub-gaussian increments required in Theorem \ref{thm: local version} (again, see \cite{V_RMT}).  
By this theorem, with probability $\geq 0.995$, 
\[\abs{Z_x} \leq C \gamma(T \cap \|x\|_2 B_2^n) \qquad \text{for all } x \in T.\]
Let $F$ be the `good' event that the above equation holds.

To control $\|A v_\lambda\|_2$, consider the `good' event $G$ that
\[\|A x\|_2 \geq \sqrt{m} \|x\|_2 - C \gamma(T \cap \|x\|_2 B_2^n) \qquad  \text{for all } x \in T.\]
By Theorem \ref{thm: local matrix version}, $G$ holds with probability at least $0.995$.  

Now, suppose that both $G$ and $F$ hold (which occurs with probability at least $0.99$ by the union bound).  We will show that for every $\lambda > 1$, $v_\lambda$ is controlled. %assuming nothing more than $G$ and $F$.  
The event $G$ gives
\[\langle v_\lambda, A^T z \rangle \leq  C \gamma(T \cap \|v_\lambda\|_2 B_2^n) \cdot \|z\|_2.\]
The event $F$ gives
\[\|A v_\lambda\|_2 \geq \sqrt{m} \|v_\lambda\|_2 - C \gamma(T \cap \|v_\lambda\|_2 B_2^n).\]
Taking square roots of both sides of \eqref{eq: starting v} and plugging in these two inequalities gives~\eqref{complicated}.\qed %to the result.  Rearrange to complete the proof.
\end{proof}

\section{Comparison with known results}
\label{sec:compare}

%\subsection{Comparison with~\cite{KM}}
%==================

Several partial cases of our main results have been known. 
As we already mentioned, the special case of Theorem~\ref{thm: deviation} 
where the entries of $A$ have standard normal distribution follows 
from the main result of the paper by G.~Schechtman \cite{Schechtman}.

Generalizing the result of \cite{Schechtman}, B.~Klartag and S.~Mendelson proved the following theorem. 

\begin{theorem}[Theorem~4.1 in~\cite{KM}]\label{thm:km4.1}
%For every $K>0$ there exists $C(K)>0$ such that the following holds.
Let $A$ be an isotropic, sub-gaussian random matrix as in \eqref{eq: A}, and let $T \subseteq S^{n-1}$.
Assume that $w(T) \ge C'(K)$.\footnote{This restriction is not explicitly mentioned in the 
  statement of Theorem~4.1 in~\cite{KM}, but it is used in the proof. Indeed, this result is 
  derived from their Theorem~1.3, which explicitly requires that $\gamma(T)$ be large enough.
  Without such requirement, Theorem~4.1 in~\cite{KM} fails e.g. when $T$ is a singleton, 
  since in that case we have $w(T)=0$.} Then with probability larger than $1/2$,
\begin{equation}
\label{km4.1}
\sup_{x \in T} \left| \norm{A x} - \sqrt m \right|
\leq C(K) w(T).
\end{equation}
Here $C'(K)$ and $C(K)$ may depend on $K$ only.
\end{theorem}

A similar but slightly more informative statement follows from our main results. 
Indeed, Corollary~\ref{cor: on sphere} gives the same conclusion, but with explicit dependence 
on $K$ (the sub-gaussian norms of the rows of $A$) as well as probability of success.
Moreover, our general results, Theorems~\ref{thm: deviation} and \ref{thm: deviation tail2}, 
do not require the set $T$ to lie on the unit sphere. 

Another related result was proved by S.~Mendelson, A.~Pajor, and N.~Tomczak-Jaegermann. 

\begin{theorem}[Theorem~2.3 in~\cite{MPT}]
\label{thm:mpt2.3}
%There exist constants $c,C>0$ for which the following holds.
Let $A$ be an isotropic, sub-gaussian random matrix as in \eqref{eq: A}, and 
$T$ be a star-shaped subset of $\R^n$.
Let $0<\theta<1$.
Then with probability at least $1 - \exp(-c \theta^2 m / K^4)$ we have that 
all vectors $x\in T$ with $$\norm{x} \geq 
r^* := \inf \left\{  
\rho>0 : \rho \geq CK^2 \gamma\left( T \cap \rho \cdot S^{n-1} \right) / (\theta \sqrt m)
\right\}$$ 
satisfy
$$
(1-\theta) \norm{x}^2 \leq \frac{\norm{Ax}^2}{m}
\leq
(1+\theta) \norm{x}^2 .$$
\end{theorem}

Applying our Theorem~\ref{thm: deviation tail2} to the bounded set 
$T \cap r^* \cdot S^{n-1}$
precisely implies Theorem~\ref{thm:mpt2.3} with the same failure probability 
(up to the values of the absolute constants $c, C$).
Moreover, our Theorem~\ref{thm: deviation tail2} 
treats all $x\in T$ uniformly, 
whereas Theorem~\ref{thm:mpt2.3} works only for $x$ with large norm.

%\begin{remark}
%Their Theorem~\ref{thm:mpt2.3} requires $T$ to be star-shaped, whereas our Theorem~\ref{thm: deviation tail2} has no such requirement.
%\end{remark}
%
%\begin{remark}
%Their Theorem~\ref{thm:mpt2.3} 
%bounds the probability of 
%$\norm{Ax}^2 - m \norm{x}^2 > 
% m (1+\theta) \norm{x}^2 $,
%whereas our Theorem~\ref{thm: deviation tail2}
%bounds the probability of 
%$\norm{Ax} - \sqrt{m}\norm{x} >
%C K^2 \big[  \gamma(T) + u \cdot \diam(T) \big]$.
%In particular, our error bound does not depend on $\norm{x}$, while theirs does depend on $\norm{x}$.
%This is why their result is only valid for $x$ with large enough norm and says nothing about $x$ with small norm,
%while our result holds for all $x\in T$.
%The bottom line is that our theorem statement is cleaner and stronger.
%\end{remark}
%
%Roman's comments:
%
%[3]: for the partial case where T is a subset of the unit sphere, the result of [3, corollary 2.6] is comparable to ours. More precisely, for expectation the results are indeed comparable, while our probability bound looks a bit better. It is not clear if in [3] one can easily get rid of the normalization assumption; perhaps not.

Yet another relevant result was proved by Dirksen \cite[Theorem 5.5]{Dirksen}.
He showed that the inequality
\begin{multline}
\label{dirksen}
\left| \|Ax\|_2^2 - m\|x\|_2^2 \right| 
  \lesssim K^2 w(T)^2 + \sqrt{m} K^2 \rad(T) w(T) \\
  + u \sqrt{m} K^2 \rad(T)^2 + u^2 K^2 \rad(T)^2
\end{multline}
holds uniformly over $x \in T$ with probability at least $1-\exp(-u^2)$.
To compare with our results, one can see that Theorem~\ref{thm: deviation tail2} implies that, 
with the same probability,
\begin{multline*}
\left| \|Ax\|_2^2 - m\|x\|_2^2 \right| 
  \lesssim K^4 w(T)^2 + \sqrt{m} K^2 \|x\|_2 w(T) \\
  + u \sqrt{m} K^2 \rad(T)\|x\|_2 + uK^4 \rad(T) w(T) + u^2 K^4 \rad(T)^2,
\end{multline*}
which is stronger than~(\ref{dirksen}) when $K = O(1)$ and $m \gtrsim n$, since then $\|x\|_2 \leq
\rad(T)$ and $w(T) \lesssim \sqrt{m} \rad(T)$.

\section{Preliminaries}
%=================

\subsection{Majorizing Measure Theorem, and deduction of Theorems~\ref{thm: deviation} and \ref{thm: deviation tail2}. }		\label{s: talagrand}
%--------------

As we mentioned in the Introduction, Theorems~\ref{thm: deviation} and \ref{thm: deviation tail2}
follow from Theorem~\ref{thm: increments} via Talagrand's Majorizing Measure Theorem (and 
its high-probability counterpart). Let us state this theorem specializing to processes that are indexed by 
points in $\R^n$.
For $T\subset \R^n$, let
$\diam(T) \coloneqq \sup_{x,y\in T}
\norm{x-y}$.
%A more general (but essentially equivalent) formulation of this theorem, which assumes only 
%that $T$ is an arbitrary semi-metric space, can be found e.g. in \cite{Talagrand}.

\begin{theorem}[Majorizing Measure Theorem]					\label{thm: talagrand}
  Consider a random process $(Z_x)_{x \in T}$ indexed by points $x$ in a bounded set $T \subset \R^n$.
  Assume that the process has sub-gaussian increments, that is there exists $M \ge 0$ such that 
  \begin{equation}
  \label{talagrand_assumption}
  \|Z_x - Z_y\|_\psitwo \le M \|x-y\|_2 \quad \textnormal{for every } x,y \in T.
  \end{equation}
  Then
  $$
  \E \sup_{x,y \in T} |Z_x - Z_y| \le C M \E \sup_{x \in T} \ip{g}{x},
  $$
  where $g \sim N(0,I_n)$.
  Moreover, for any $u \ge 0$, the event 
  $$
  \sup_{x,y \in T} |Z_x-Z_y| \le C M \big[ \E \sup_{x \in T} \ip{g}{x} + u \diam(T) \big]
  $$
  holds with probability at least $1-\exp(-u^2)$.
\end{theorem}

The first part of this theorem can be found e.g. in \cite[Theorems~2.1.1, 2.1.5]{Talagrand}. 
The second part, a high-probability bound, is borrowed from \cite[Theorem~3.2]{Dirksen}.

\bigskip 

Let us show how to deduce Theorems~\ref{thm: deviation} and \ref{thm: deviation tail2}. 
According to Theorem~\ref{thm: increments}, the random process $Z_x := \|Ax\|_2 -  \sqrt{m} \|x\|_2 $ 
satisfies the hypothesis \eqref{talagrand_assumption} of the Majorizing Measure Theorem~\ref{thm: talagrand} with $M=CK^2$.
Fix an arbitrary $y \in T$ and use the triangle inequality to obtain
\begin{equation}\label{eq: important label}
\E \sup_{x \in T} |Z_x| \le \E \sup_{x \in T} |Z_x - Z_y| + \E |Z_y|. 
\end{equation}
Majorizing Measure Theorem bounds the first term: $\E \sup_{x \in T} |Z_x - Z_y| \lesssim K^2 w(T)$.
(We suppress absolute constant factors in this inequality and below.)
The second term can be bounded more easily as follows:
$\E |Z_y| \lesssim \|Z_y\|_\psitwo \lesssim K^2 \|y\|_2$, where we again used Theorem~\ref{thm: increments}
with $x=0$. Using \eqref{eq: complexity vs width}, we conclude that 
$$
\E \sup_{x \in T} |Z_x| \lesssim K^2 (w(T) + \|y\|_2) \lesssim K^2 \gamma(T),
$$
as claimed in Theorem~\ref{thm: deviation}.

We now prove Theorem~\ref{thm: deviation tail2}.
Since adding $0$ to a set does not change its radius, we may assume that $0 \in T$.
Let $Z_x := \norm{Ax} - \sqrt{m}\norm{x}$.
Since $Z_0=0$, and since $Z_x$ has sub-gaussian increments by Theorem~\ref{thm: increments}, Theorem~\ref{thm: talagrand} gives that with probability at least $1-\exp(-u^2)$,
\begin{align*}
\sup_{x\in T} \abs{Z_x}
= 
\sup_{x\in T} \abs{Z_x-Z_0}
& \lesssim 
K^2 \big[  \E \sup_{x \in T} \ip{g}{x} + u \cdot \diam(T) \big] \\
& \lesssim 
K^2 \big[  \E \sup_{x \in T} \ip{g}{x} + u \cdot \rad(T) \big].\qed
\end{align*}

% Theorem~\ref{thm: deviation tail} follows similarly, by using the high-probability part of Majorizing Measure Theorem~\ref{thm: talagrand}. 

\subsection{Sub-exponential random variables, and Bernstein's inequality}
%==================

Our argument will make an essential use of Bernstein's inequality for sub-exponential 
random variables. Let us briefly recall the relevant notions, which can be found, e.g., in \cite{V_RMT}.
A random variable $Z$ is {\em sub-exponential} if its distribution is dominated by an 
exponential distribution. More formally, $Z$ is sub-exponential if the Orlicz norm 
$$
\|Z\|_\psione := \inf \big\{ K>0 :\; \E \psi_1(|Z|/K) \le 1 \} 
$$
is finite, for the Orlicz function $\psi_1(x) = \exp(x) - 1$.  
Every sub-gaussian random variable is sub-exponential. Moreover, an application of Young's inequality 
implies the following relation for any two sub-gaussian random variables $X$ and $Y$:
\begin{equation}         \label{eq: XY}
\|XY\|_\psione \le \|X\|_\psitwo \|Y\|_\psitwo.
\end{equation}

%{\bf Inequality \eqref{eq: XY} was noted in \cite{Dirksen} and should be checked carefully. I am not suggesting to include the proof here, but maybe Chris can check and send us a proof separately?}

The classical Bernstein's inequality states that a sum of independent sub-exponential random variables
is dominated by a mixture of sub-gaussian and sub-exponential distributions. 

\begin{theorem}[Bernstein-type deviation inequality, see e.g. \cite{V_RMT}]			\label{thm: bernstein}
  Let $X_1,\ldots,X_m$ be independent random variables, 
  which satisfy $\E X_i = 0$ and $\|X_i\|_\psione \le L$.
  Then
  $$
  \Pr{ \Big| \frac{1}{m} \sum_{i=1}^m X_i \Big| > t } 
  \le 2 \exp \Big[ - c m \min \Big(\frac{t^2}{L^2}, \, \frac{t}{L} \Big) \Big], \quad t \ge 0.
  $$
\end{theorem}

\section{Proof of Theorem~\ref{thm: increments}}					\label{s: increments proof}
%=====================

\begin{proposition}[Concentration of the norm]				\label{prop: conc norm}
  Let $X \in \R^m$ be a random vector with independent coordinates $X_i$ that satisfy
  $\E X_i^2 = 1$ and $\|X_i\|_\psitwo \le K$. Then 
  $$
  \Big\| \|X\|_2 - \sqrt{m} \Big\|_\psitwo \le CK^2.
  $$
\end{proposition}

\begin{remark}
If $\E X_i=0$, this proposition follows from~\cite[Theorem~2.1]{RV},
whose proof uses the Hanson-Wright inequality.
\end{remark}

\begin{proof}
Let us apply Bernstein's deviation inequality (Theorem~\ref{thm: bernstein}) for 
the sum of independent random variables 
$\|X\|_2^2 - m = \sum_{i=1}^m (X_i^2-1)$. These random variables have zero means and sub-exponential norms 
$$
\|X_i^2 - 1\|_\psione \le 2 \|X_i^2\|_\psione \le 2 \|X_i\|_\psitwo^2 \le 2K^2.
$$  
(Here we used a simple centering inequality which can be found e.g. in \cite[Remark~5.18]{V_RMT} and the
inequality \eqref{eq: XY}.) Bernstein's inequality implies that
\begin{equation}         \label{eq: conc X2}
\Pr{ \big| \|X\|_2^2 - m \big| > tm } 
  \le 2 \exp \Big[ - c m \min \Big(\frac{t^2}{K^4}, \, \frac{t}{K^2} \Big) \Big], \quad t \ge 0.
\end{equation}

To deduce a concentration inequality for $\|X\|_2 - \sqrt{m}$  from this, let us
employ the numeric bound $|x^2-m| \ge \sqrt{m}\, |x-\sqrt{m}|$ valid for all $x \ge 0$. 
Using this together with \eqref{eq: conc X2} for $t=s/\sqrt{m}$, we obtain
\begin{align*}
\Pr{ \big| \|X\|_2 - \sqrt{m} \big| > s } 
  &\le \Pr{ \big| \|X\|_2^2 - m \big| > s \sqrt{m}} \\
  &\le 2 \exp(-cs^2/K^4) \quad \text{for } s \le K^2 \sqrt{m}.
\end{align*}
To handle large $s$, we proceed similarly but with a different numeric bound, namely 
$|x^2-m| \ge (x-\sqrt{m})^2$ which is valid for all $x \ge 0$.
Using this together with \eqref{eq: conc X2} for $t=s^2/m$, we obtain
\begin{align*}
\Pr{ \big| \|X\|_2 - \sqrt{m} \big| > s } 
  &\le \Pr{ \big| \|X\|_2^2 - m \big| > s^2} \\
  &\le 2 \exp(-cs^2/K^2) \quad \text{for } s \ge K \sqrt{m}.
\end{align*}
Since $K \ge 1$,  in both cases we bounded the probability in question by 
$2 \exp(-cs^2/K^4)$. This completes the proof.\qed 
\end{proof}

\begin{lemma}[Concentration of a random matrix on a single vector]		\label{lem: conc Ax}
  Let $A$ be an isotropic, sub-gaussian random matrix as in \eqref{eq: A}. Then 
  \[
  \Big\| \|Ax\|_2 - \sqrt{m} \Big\|_\psitwo \le CK^2 \quad \mathrm{for\ every\ } x \in S^{n-1}.
  \]
\end{lemma}

\begin{proof}
The coordinates of the vector $Ax \in \R^m$ are independent random variables $X_i := \ip{A_i}{x}$.
The assumption that $\E A_i A_i^\tran = I$ implies that $\E X_i^2 = 1$, and 
the assumption that $\|A_i\|_\psitwo \le K$ implies that $\|X_i\|_\psitwo \le K$.
The conclusion of the lemma then follows from Proposition~\ref{prop: conc norm}. \qed
\end{proof}

Lemma~\ref{lem: conc Ax} can be viewed as a partial case of the increment inequality of 
Theorem~\ref{thm: increments} for $x \in S^{n-1}$ and $y=0$, namely 
\begin{equation}         \label{eq: increments y=0}
\|Z_x\|_\psitwo \le CK^2 \quad \text{for every } x \in S^{n-1}.
\end{equation}
Our next intermediate step is to extend this by allowing $y$ to be an arbitrary unit vector.

\begin{lemma}[Sub-gaussian increments for unit vectors]					\label{lem: increments unit}			
  Let $A$ be an isotropic, sub-gaussian random matrix as in \eqref{eq: A}. Then
  $$
  \Big\| \|Ax\|_2 -\|Ay\|_2 \Big\|_\psitwo \le CK^2 \|x-y\|_2 \quad \textnormal{for every } x,y \in S^{n-1}.
  $$
\end{lemma}

\begin{proof}
Given $s \ge 0$, we will bound the tail probability 
\begin{equation}         \label{eq: p}
p:= \Pr{ \frac{\big| \|Ax\|_2-\|Ay\|_2 \big|}{\|x-y\|_2} > s}.
\end{equation}

{\em Case 1: $s \ge 2\sqrt{m}$.} Using the triangle inequality we have $|\|Ax\|_2-\|Ay\|_2| \le \|A(x-y)\|_2$.
Denoting $u := (x-y)/\|x-y\|_2$, we find that 
$$
p \le \Pr{\|Au\|_2 > s} 
\le \Pr{\|Au\|_2 -\sqrt{m} > s/2} 
\le \exp(-Cs^2/K^4).
$$
Here the second bound holds since $s \ge 2\sqrt{m}$, and the last bound follows by Lemma~\ref{lem: conc Ax}.

\medskip

{\em Case 2: $s \le 2\sqrt{m}$.} Multiplying both sides of the inequality defining $p$ in \eqref{eq: p} 
by $\|Ax\|_2+\|Ay\|_2$, we can write $p$ as
$$
p = \Pr{ |Z| > s \big( \|Ax\|_2+\|Ay\|_2 \big) }
\quad \text{where} \quad
Z := \frac{\|Ax\|_2^2-\|Ay\|_2^2}{\|x-y\|_2}.
$$
In particular,
$$
p \le \Pr{ |Z| > s \|Ax\|_2} \le \Pr{ |Z| > \frac{s\sqrt{m}}{2} } + \Pr{ \|Ax\|_2 \le \frac{\sqrt{m}}{2} }
=: p_1 + p_2.
$$

We may bound $p_2$ using Lemma~\ref{lem: conc Ax}: 
\begin{equation}         \label{eq: p2}
p_2 \le 2\exp \Big( - \frac{(\sqrt{m}/2)^2}{C^2 K^4} \Big) 
= 2\exp \Big( - \frac{m}{4 C^2 K^4} \Big) 
\le 2\exp \Big( - \frac{s^2}{16 C^2 K^4} \Big). 
\end{equation}

Next, to bound $p_1$, it will be useful to write $Z$ as
$$
Z = \frac{\ip{A(x-y)}{A(x+y)}}{\|x-y\|_2} = \ip{Au}{Av}, 
\quad \text{where} \quad
u := \frac{x-y}{\|x-y\|_2}, \quad 
v := x+y.
$$
Since the coordinates of $Au$ and $Av$ are $\ip{A_i}{u}$ and $\ip{A_i}{v}$ respectively, 
$Z$ can be represented as a sum of independent random variables:
\begin{equation}         \label{eq: Z sum}
Z = \sum_{i=1}^m \ip{A_i}{u} \ip{A_i}{v}.
\end{equation}

Note that each of these random variables $\ip{A_i}{u} \ip{A_i}{v}$ has zero mean, since 
$$
\E \ip{A_i}{x-y} \ip{A_i}{x+y} = \E \big[ \ip{A_i}{x}^2 - \ip{A_i}{y}^2 \big]
= 1 - 1 = 0.
$$
(Here we used the assumptions that $\E A_i A_i^\tran = I$ and $\|x\|_2 = \|y\|_2 = 1$.)
Moreover, the assumption that $\|A_i\|_\psitwo \le K$ implies that 
$\|\ip{A_i}{u}\|_\psitwo \le K \|u\|_2 = K$ and $\|\ip{A_i}{v}\|_\psitwo \le K \|v\|_2 \le 2K$.
Recalling inequality \eqref{eq: XY}, we see that $\ip{A_i}{u} \ip{A_i}{v}$ are sub-exponential 
random variables with $\|\ip{A_i}{u} \ip{A_i}{v}\|_\psione \le C K^2$.
Thus we can apply Bernstein's inequality (Theorem~\ref{thm: bernstein}) to the 
sum of mean zero, sub-exponential random variables in \eqref{eq: Z sum}, and obtain 
$$
p_1 = \Pr{ |Z| > \frac{s\sqrt{m}}{2} }
\le 2\exp(-cs^2/K^4), \quad \text{since } s \le 2 K^2 \sqrt{m}.
$$

Combining this with the bound on $p_2$ obtained in \eqref{eq: p2}, we conclude that
$$
p = p_1 + p_2 \le 2 \exp(-cs^2/K^4).
$$
This completes the proof.\qed
\end{proof}

\medskip

Finally, we are ready to prove the increment inequality in full generality, for all $x,y \in \R^n$. 

\begin{proof}[of Theorem~\ref{thm: increments}]
Without loss of generality we may assume that $\|x\|_2 = 1$ and $\|y\|_2 \ge 1$. 
Consider the unit vector $\bar{y} := y/\|y\|_2$ and apply the triangle inequality to get
$$
\|Z_x - Z_y\|_\psitwo \le \|Z_x - Z_{\bar{y}}\|_\psitwo +  \|Z_{\bar{y}} - Z_y\|_\psitwo =: R_1 + R_2.
$$
By Lemma~\ref{lem: increments unit}, $R_1 \le CK^2 \|x-\bar{y}\|_2$. 
Next, since $\bar{y}$ and $y$ are collinear, we have $R_2 = \|\bar{y} - y\|_2 \cdot \|Z_{\bar{y}}\|_\psitwo$.
Since $\bar{y}\in S^{n-1}$, inequality \eqref{eq: increments y=0} states that $\|Z_{\bar{y}}\|_\psitwo \le C K^2$, 
and we conclude that $R_2 \le CK^2 \|\bar{y} - y\|_2$. Combining the bounds on $R_1$ and $R_2$, 
we obtain 
$$
\|Z_x - Z_y\|_\psitwo \le CK^2 \big( \|x-\bar{y}\|_2 + \|\bar{y} - y\|_2 \big).
$$
It is not difficult to check that since $\|y\|_2 \ge 1$, we have $\|x-\bar{y}\|_2 \le \|x-y\|_2$ and 
$\|\bar{y} - y\|_2 \le \|x-y\|_2$. This completes the proof. \qed
\end{proof}

\section{Proof of Theorem~\ref{thm: local version}}		\label{sec:prooflocal}
%==================

We will prove a slightly stronger statement. For $r>0$, define
$$
E_r := \sup_{x \in \frac{1}{r} T \cap B_2^n} |Z_x|.
$$
Set $W := \lim_{r \to \rad(T)_-} \gamma\left( \frac{1}{r} T \cap B_2^n \right)$.
Since $\frac{1}{r} T \cap B_2^n$ contains at least one point on the boundary for every $r < \rad(T)$, 
it follows that $W \geq \sqrt{2/\pi}$.
We will show that, with probability at least $1 - \exp\left( -c't^2 W^2 \right)$, one has
$$
E_r \leq t \cdot CK^2 \gamma\left(
\frac{1}{r} T \cap B_2^n \right) \mathrm{ for\  all\ } r \in (0, \infty),
$$ 
which, when combined with the assumption of homogeneity, will clearly imply the theorem with a stronger probability. 

Fix $\eps > 0$. Let $\eps = r_0 < r_1 < \ldots < r_N$ be a sequence of real
numbers satisfying the following conditions:
\begin{itemize}
   \item $\gamma\left( \frac{1}{r_i} T \cap B_2^n \right) = 2 \cdot
      \gamma\left( \frac{1}{r_{i+1}} T \cap B_2^n \right)$ for $i=0,1,\dots,N-1$, and
   \item
      $\gamma\left( \frac{1}{r_N} T \cap B_2^n \right) \leq 2 \cdot W$.
\end{itemize}
The quantities $r_1, \ldots, r_N$ exist since the map $r \mapsto
\gamma\left( \frac{1}{r} T \cap B_2^n \right)$ is decreasing and
continuous when $T$ is star-shaped.

Applying the Majorizing Measure Theorem~\ref{thm: talagrand} to the set $\frac{1}{r} T \cap B_2^n$
and noting that $Z_0=0$, we obtain that
$$
E_r \lesssim K^2 \left[ \gamma\left( \frac{1}{r} T \cap B_2^n \right) + u \right]
$$
with probability at least $1-\exp(-u^2)$.
Set $c \coloneqq 10 \cdot \sqrt{\frac{\pi}{2}} \geq 10/W$ and use the above inequality 
for $u = c t \gamma\left( \frac{1}{r} T \cap B_2^n \right)$. We get
\begin{equation}
   E_r \lesssim t\cdot K^2 \gamma\left( \frac{1}{r} T \cap B_2^n \right)
   \label{eqn:concr}
\end{equation}
holds with probability at
least $1 - \exp\left( - c^2 t^2\gamma\left( \frac{1}{r} T \cap B_2^n \right)^2
\right)$. Thus for each $i \in \{0,1,\ldots, N\}$, we have
\begin{equation}
   E_{r_i} \lesssim 
   t\cdot K^2 \gamma \left( \frac{1}{r_i} T \cap B_2^n \right)
   \label{eqn:conc}
\end{equation}
with probability at least $$1 - \exp\left( -c^2 t^24^{N-i} \gamma \left( \frac{1}{r_N} T
\cap B_2^n \right)^2 \right) \geq 1 - \exp\left( -c^2 t^2 4^{N-i} W^2\right).$$
By our choice of $c$ and the union bound, \eqref{eqn:conc} holds for all $i$
simultaneously with probability at least
\[
   1 - \sum_{i=0}^N \exp\left( -c^2 t^24^{N-i} W^2 \right)
   \geq 1 - 2 \cdot \exp(-100 t^2 W^2) \eqqcolon 1-\exp(-c't^2 W^2).
\]

We now show that if \eqref{eqn:conc} holds for all $i$, then
\eqref{eqn:concr} holds for all $r \in (\eps, \infty)$. This is done via an
approximation argument. To this end, assume that \eqref{eqn:conc} holds and let
$r \in (r_{i-1}, r_{i})$ for some $i \in [N]$.  Since $T$ is star-shaped, we have
$\frac{1}{r} T \cap B_2^n \subseteq \frac{1}{r_{i-1}} T \cap B_2^n$, so
\begin{align*}
   E_r 
   \leq E_{r_{i-1}} 
   \lesssim t\cdot K^2 \gamma\left( \frac{1}{r_{i-1}} T \cap
   B_2^n\right)
   & = 2 t\cdot K^2 \gamma\left( \frac{1}{r_{i}} T \cap B_2^n \right) \\
   & \leq 2t\cdot K^2 \gamma\left( \frac{1}{r} T \cap B_2^n \right).
\end{align*}
Also, for $\rad(T) \geq r > r_N$ we have
\[
E_{r} 
\lesssim t\cdot K^2 \gamma\left( \frac{1}{r_N} T \cap B_2^n \right) 
\leq 2t\cdot K^2 W
\leq 2t \cdot K^2 \gamma\left( \frac{1}{r} T \cap B_2^n \right).
\]

Let $F_k$ be the event that \eqref{eqn:concr} holds for all $r \in (1/k,
\infty)$. We have just shown that $\Pr{F_k} \geq 1 - \exp\left( -c't^2W^2
\right)$ for all $k \in \N$. As $F_1 \supseteq F_2 \supseteq \ldots$ and
$\cap_k F_k \eqqcolon F_{\infty}$ is the event that \eqref{eqn:concr} holds for
all $r \in (0,\infty)$, it follows by continuity of measure that $\Pr{F_\infty}
\geq 1 - \exp\left( -c't^2W^2 \right)$, thus completing the proof.

\section{Further Thoughts}
%===============

In the definition of Gaussian complexity $\gamma(T) = \E \sup_{x \in T} |\ip{g}{x}|$, 
the absolute value is essential to make Theorem~\ref{thm: deviation} hold. 
In other words, the bound would fail if we replace $\gamma(T)$ by the 
Gaussian width $w(T) = \E \sup_{x \in T} \ip{g}{x}$. 
This can be seen by considering a set $T$ that consists of a single point. 

However, {\em one-sided} deviation inequalities do hold for Gaussian width. 
Thus a one-sided version of Theorem~\ref{thm: deviation} states that 
\begin{equation}
\E \sup_{x \in T} \Big( \|Ax\|_2 -  \sqrt{m} \|x\|_2 \Big) \le C K^2 \cdot w(T),
\label{onesided}
\end{equation}
and the same bound holds for $\E \sup_{x \in T} \big( -\|Ax\|_2 +  \sqrt{m} \|x\|_2  \big)$.
To prove~(\ref{onesided}), one modifies 
the argument in Section~\ref{s: talagrand} as follows.
Fix a $y\in T$.
Since 
$\E \|Ay\|_2 \le \left( \E \|Ay\|_2^2 \right)^{1/2} = \sqrt{m} \|y\|_2$,
we have $\E Z_y \leq 0$, thus
$$
\E \sup_{x \in T} Z_x 
\leq \E \sup_{x \in T} (Z_x - Z_y) 
\le \E \sup_{x \in T} |Z_x - Z_y|
\lesssim K^2 w(T)
$$
where the last bound follows by Majorizing Measure Theorem~\ref{thm: talagrand}. 
Thus in this argument there is no need to separate the term $\E|Z_y|$ as was done before in Equation \eqref{eq: important label}. 

\bibliographystyle{spmpsci}
\bibliography{lmpv-deviation-bib}

\begin{thebibliography}{10}
\providecommand{\url}[1]{{#1}}
\providecommand{\urlprefix}{URL }
\expandafter\ifx\csname urlstyle\endcsname\relax
  \providecommand{\doi}[1]{DOI~\discretionary{}{}{}#1}\else
  \providecommand{\doi}{DOI~\discretionary{}{}{}\begingroup
  \urlstyle{rm}\Url}\fi

\bibitem{amelunxen2014living}
Amelunxen, D., Lotz, M., McCoy, M.B., Tropp, J.A.: Living on the edge: phase
  transitions in convex programs with random data.
\newblock Inf. Inference \textbf{3}(3), 224--294 (2014).
\newblock \doi{10.1093/imaiai/iau005}.
\newblock \urlprefix\url{http://dx.doi.org/10.1093/imaiai/iau005}

\bibitem{AGM}
Artstein-Avidan, S., Giannopoulos, A., Milman, V.D.: Asymptotic geometric
  analysis. {P}art {I}, \emph{Mathematical Surveys and Monographs}, vol. 202.
\newblock American Mathematical Society, Providence, RI (2015)

\bibitem{chandrasekaran2012convex}
Chandrasekaran, V., Recht, B., Parrilo, P.A., Willsky, A.S.: The convex
  geometry of linear inverse problems.
\newblock Foundations of Computational mathematics \textbf{12}(6), 805--849
  (2012)

\bibitem{Dirksen}
Dirksen, S.: Tail bounds via generic chaining.
\newblock Electron. J. Probab. \textbf{20}(53), 1--29 (2015).
\newblock \doi{10.1214/EJP.v20-3760}.
\newblock \urlprefix\url{http://dx.doi.org/10.1214/EJP.v20-3760}

\bibitem{eldar2012compressed}
Eldar, Y.C., Kutyniok, G.: Compressed sensing: theory and applications.
\newblock Cambridge University Press (2012)

\bibitem{foucart2013mathematical}
Foucart, S., Rauhut, H.: A mathematical introduction to compressive sensing.
\newblock Applied and Numerical Harmonic Analysis. Birkh\"auser/Springer, New
  York (2013).
\newblock \doi{10.1007/978-0-8176-4948-7}.
\newblock \urlprefix\url{http://dx.doi.org/10.1007/978-0-8176-4948-7}

\bibitem{gordon1988milman}
Gordon, Y.: On {M}ilman's inequality and random subspaces which escape through
  a mesh in {${\bf R}\sp n$}.
\newblock In: Geometric aspects of functional analysis (1986/87), \emph{Lecture
  Notes in Math.}, vol. 1317, pp. 84--106. Springer, Berlin (1988).
\newblock \doi{10.1007/BFb0081737}.
\newblock \urlprefix\url{http://dx.doi.org/10.1007/BFb0081737}

\bibitem{gross2010quantum}
Gross, D., Liu, Y.K., Flammia, S.T., Becker, S., Eisert, J.: Quantum state
  tomography via compressed sensing.
\newblock Physical review letters \textbf{105}(15), 150,401 (2010)

\bibitem{JL}
Johnson, W.B., Lindenstrauss, J.: Extensions of lipschitz mappings into a
  hilbert space.
\newblock Contemporary mathematics \textbf{26}(189-206), 1 (1984)

\bibitem{KM}
Klartag, B., Mendelson, S.: Empirical processes and random projections.
\newblock J. Funct. Anal. \textbf{225}(1), 229--245 (2005).
\newblock \doi{10.1016/j.jfa.2004.10.009}.
\newblock \urlprefix\url{http://dx.doi.org/10.1016/j.jfa.2004.10.009}

\bibitem{lecue2013learning}
Lecu{\'e}, G., Mendelson, S.: Learning subgaussian classes: Upper and minimax
  bounds  (2013).
\newblock Available at \url{http://arxiv.org/abs/1305.4825}

\bibitem{lustig2007sparse}
Lustig, M., Donoho, D., Pauly, J.M.: Sparse {M}{R}{I}: The application of
  compressed sensing for rapid {M}{R} imaging.
\newblock Magnetic resonance in medicine \textbf{58}(6), 1182--1195 (2007)

\bibitem{MPT}
Mendelson, S., Pajor, A., Tomczak-Jaegermann, N.: Reconstruction and
  subgaussian operators in asymptotic geometric analysis.
\newblock Geom. Funct. Anal. \textbf{17}(4), 1248--1282 (2007).
\newblock \doi{10.1007/s00039-007-0618-7}.
\newblock \urlprefix\url{http://dx.doi.org/10.1007/s00039-007-0618-7}

\bibitem{milman1985geometrical}
Milman, V.D.: Geometrical inequalities and mixed volumes in the local theory of
  banach spaces.
\newblock Ast{\'e}risque \textbf{131}, 373--400 (1985)

\bibitem{milman1985random}
Milman, V.D.: Random subspaces of proportional dimension of finite dimensional
  normed spaces: approach through the isoperimetric inequality.
\newblock In: Banach spaces, pp. 106--115. Springer (1985)

\bibitem{oymak2013simple}
Oymak, S., Thrampoulidis, C., Hassibi, B.: Simple bounds for noisy linear
  inverse problems with exact side information  (2013).
\newblock Available at \url{http://arxiv.org/abs/1312.0641}

\bibitem{oymak2013squared}
Oymak, S., Thrampoulidis, C., Hassibi, B.: The squared-error of generalized
  lasso: A precise analysis.
\newblock In: 51st Annual Allerton Conference on Communication, Control, and
  Computing, pp. 1002--1009. IEEE (2013)

\bibitem{pajor1986subspaces}
Pajor, A., Tomczak-Jaegermann, N.: Subspaces of small codimension of
  finite-dimensional banach spaces.
\newblock Proceedings of the American Mathematical Society \textbf{97}(4),
  637--642 (1986)

\bibitem{PV}
Plan, Y., Vershynin, R.: Robust 1-bit compressed sensing and sparse logistic
  regression: a convex programming approach.
\newblock IEEE Trans. Inform. Theory \textbf{59}(1), 482--494 (2013).
\newblock \doi{10.1109/TIT.2012.2207945}.
\newblock \urlprefix\url{http://dx.doi.org/10.1109/TIT.2012.2207945}

\bibitem{plan2015generalized}
Plan, Y., Vershynin, R.: The generalized lasso with non-linear observations.
\newblock IEEE Transactions on Information Theory \textbf{62}(3), 1528--1537
  (2016).
\newblock \doi{10.1109/TIT.2016.2517008}

\bibitem{plan2014high}
Plan, Y., Vershynin, R., Yudovina, E.: High-dimensional estimation with
  geometric constraints  (2014).
\newblock Available at \url{http://arxiv.org/abs/1404.3749}

\bibitem{RV}
Rudelson, M., Vershynin, R.: Hanson-{W}right inequality and sub-{G}aussian
  concentration.
\newblock Electron. Commun. Probab. \textbf{18}(82), 1--9 (2013).
\newblock \doi{10.1214/ECP.v18-2865}.
\newblock \urlprefix\url{http://dx.doi.org/10.1214/ECP.v18-2865}

\bibitem{Schechtman}
Schechtman, G.: Two observations regarding embedding subsets of {E}uclidean
  spaces in normed spaces.
\newblock Adv. Math. \textbf{200}(1), 125--135 (2006).
\newblock \doi{10.1016/j.aim.2004.11.003}.
\newblock \urlprefix\url{http://dx.doi.org/10.1016/j.aim.2004.11.003}

\bibitem{Talagrand}
Talagrand, M.: The generic chaining: Upper and lower bounds of stochastic
  processes.
\newblock Springer Monographs in Mathematics. Springer-Verlag, Berlin (2005)

\bibitem{thrampoulidis2014simple}
Thrampoulidis, C., Oymak, S., Hassibi, B.: Simple error bounds for regularized
  noisy linear inverse problems.
\newblock In: IEEE International Symposium on Information Theory (ISIT), pp.
  3007--3011. IEEE (2014)

\bibitem{vershynin2012close}
Vershynin, R.: How close is the sample covariance matrix to the actual
  covariance matrix?
\newblock Journal of Theoretical Probability \textbf{25}(3), 655--686 (2012)

\bibitem{V_RMT}
Vershynin, R.: Introduction to the non-asymptotic analysis of random matrices.
\newblock In: Compressed sensing, pp. 210--268. Cambridge Univ. Press,
  Cambridge (2012)

\bibitem{V_Estimation}
Vershynin, R.: Estimation in high dimensions: a geometric perspective.
\newblock In: Sampling Theory, a Renaissance, pp. 3--66. Birkhauser Basel
  (2015)

\bibitem{von1961collected}
Von~Neumann, J.: Collected works.
\newblock Oxford: Pergamon, 1961, edited by Taub, AH  (1961)

\end{thebibliography}

\end{document}